\documentclass[leqno]{amsart}
\usepackage{amsthm,amsfonts,amsmath,mathptmx}

\makeatletter \@namedef{subjclassname@2010}{%
2010 Mathematics Subject Classification} \makeatother

\newtheorem{theorem}{Theorem}

\newtheorem{lemma}{Lemma}
\newtheorem{remark}{Remark}
\newtheorem{example}{Example}

\theoremstyle{definition}
\newtheorem{definition}{Definition}

\theoremstyle{remark}

\thanks{The second author is thankful to School of Mathematical
Sciences, Universiti Sains Malaysia  and in particular to the
Research Group in Geometric Function Theory for permission to visit
the  group and for the helpful discussions during the preparation of
this manuscript.}

\begin{document}

\title{Coefficient Estimates for Meromorphic Bi-Univalent Functions}

\author{Suzeini Abd Halim}
\address{Institute of Mathematical Science\\
Faculty of Science, University of Malaya\\
50603 Kuala Lumpur, Malaysia} \email{suzeini@um.edu.my}

\author{Samaneh G. Hamidi}
\address{Institute of Mathematical Science\\
Faculty of Science, University of Malaya\\
50603 Kuala Lumpur, Malaysia} \email{s.hamidi\_61@yahoo.com}

\author{V. Ravichandran}
\address{Department of Mathematics\\ University of Delhi,
Delhi 110 007, India,  and \\
School of Mathematical Sciences\\ Universiti Sains Malaysia, 11800
USM, Penang, Malaysia} \email{vravi@maths.du.ac.in}

\begin{abstract}
A univalent meromorphic function defined on $\Delta:= \{
  z \in \mathbb{C}: 1<|z|<\infty \}$ with univalent inverse defined on $\Delta$
is bi-univalent meromorphic in $\Delta$. For certain subclasses of
meromorphic bi-univalent functions, estimates on the initial
coefficients are obtained.\end{abstract}



\maketitle

\section{Introduction}
An analytic function defined  on some open set $D$ that maps
different points of $D$ to different points is called univalent in
$D$ and let $\mathcal {S}$ denote the class of univalent functions
$f$ defined on the open unit disk $\mathbb{D}:= \{
  z \in \mathbb{C}:~|{\it z}|<1 \}$ of the form
\begin{equation}\label{eq1.1}
f(z)= z + \sum_{k=2}^\infty a_k z^k.
\end{equation}
The well-known Koebe one-quarter theorem asserts that the function
$f\in\mathcal{S}$ has an inverse defined on  disk $\mathbb{D_{\rho}}
:= \{z:z \in \mathbb{C}~~ {\rm{and}}~~ |{\it z}|< \rho\}$, $(\rho
\geq \frac{1}{4})$. Thus, the inverse of $f \in \mathcal{S}$ is a
univalent analytic function on  the disk $\mathbb{D_{\rho}}$. The
function $f \in \mathcal{S}$ is called \emph{bi-univalent} in
$\mathbb{D}$ if $f^{-1}$ is also univalent in the whole disk
$\mathbb{D}$. The class  $\mathcal{\sigma}$ of bi-univalent analytic
functions was introduced in 1967 by Lewin \cite{Lewin} and he showed
that, for every functions $f \in \mathcal{\sigma}$ of the form
\eqref{eq1.1}, the second coefficient of $f$ satisfy the inequality
$|a_2|< 1.51$. Subsequently, Brannan and Clunie \cite{BrannanClunie}
improved Lewin's result by showing $|a_2| \leq \sqrt {2}$. Later,
Netanyahu \cite{Netanyahu} proved that $\max_{f \in~\sigma}~|a_2| =
4/3$. Since then, several authors such as Brannan and Taha
\cite{BrannanTaha}, Taha \cite{Taha} investigated the subclasses of
bi-univalent analytic functions and found estimates on the initial
coefficients for functions in these subclasses. Recently, Ali
\emph{et al.} \cite{Ali}, Frasin and Aouf \cite{FrasinAouf},
Srivastava \emph{et al.} \cite{Srivastava} also introduced new
subclasses of bi-univalent functions and found estimates on the
coefficients $a_2$ and $a_3$ for functions in these classes.

In this paper, the concept of bi-univalency is extended to the class
of meromorphic functions defined on  $\Delta :=\{z: z \in
\mathbb{C}~~{\rm and }~~1<|{\it z}|<\infty \}$. For this purpose,
let $\Sigma$ denote the class of all meromorphic univalent functions
$g$ of the form
\begin{equation}\label{eq1.2}
g(z)= z + \sum_{n=0}^{\infty} \frac{b_n}{z^n},
\end{equation}
defined on the domain $\Delta$. Since $g \in \Sigma$ is univalent,
it has an inverse $g^{-1}$ that satisfy
\begin{equation*}
g^{-1}(g(z))=z \quad (z \in \Delta),
\end{equation*}
and
\begin{equation*}
g(g^{-1}(w))=w \quad \left( M<|w|< \infty, ~ M>0 \right).
\end{equation*}
Furthermore, the inverse function $g^{-1}$ has a series expansion of
the form
\begin{equation}\label{eq1.3}
g^{-1}(w)= w+ \sum_{n=0}^{\infty}\frac{B_n}{w^n},
\end{equation}
where $M<|w|< \infty$. Analogous to the bi-univalent analytic
functions, a function $g \in \Sigma$ is said to be \emph{meromorphic
bi-univalent} if $g^{-1} \in \Sigma$. The class of all meromorphic
bi-univalent functions is denote by $\Sigma _{\mathcal{B}}$.

Estimates on the coefficient of meromorphic univalent functions were
investigated in the literature; for example, Schiffer
\cite{Schiffer} obtained the estimate $|b_2| \leq 2/3$ for
meromorphic univalent functions $g\in\Sigma$ with $b_0=0$. In 1971,
Duren \cite{Duren} gave an elementary proof of the inequality $|b_n|
\leq 2/ (n+1)$ on the coefficient of meromorphic univalent functions
$g \in \Sigma$ with $b_k=0$ for $1 \leq k < n/2$. For the
coefficient of the inverse of meromorphic univalent functions,
Springer \cite{Springer} proved that
\begin{equation*}
|B_3| \leq 1 \quad  \text{and} \quad  |B_3+ \frac{1}{2} B_1^2| \leq
\frac{1}{2},
\end{equation*}
and conjectured that
\begin{equation*}
|B_{2n-1}| \leq \frac{(2n-2)!}{n!(n-1)!} \quad    (n=1, 2, ...).
\end{equation*}
In 1977, Kubota \cite{Kubota} has proved that the Springer
conjecture is true for $n=3,4,5$ and subsequently Schober
\cite{Schober} obtained a sharp bounds for the coefficients
$B_{2n-1}$, $1 \leq n \leq 7$, of the inverse of meromorphic
univalent functions in $\Delta$. Recently, Kapoor and Mishra
\cite{KapoorMishra} found the coefficient estimates for a class
consisting of inverses of meromorphic  starlike univalent functions
of order $\alpha$ in $\mathbb{D}$.

In the present investigation, certain subclasses of meromorphic
bi-univalent functions are introduced and estimates for the
coefficients $b_0$ and $b_1$ of functions in the newly introduced
subclasses are obtained. These coefficients results are obtained by
associating the given functions with the  functions having positive
real part. An analytic  function $p$  of the form $p(z) = 1+ c_1 z +
c_2 z^2 + \cdots$ is called a \emph{function with positive real
part} in $ \mathbb{D}$ if $\operatorname{Re} p(z)>0$ for all $z \in
\mathbb{D}$.   The class of all functions with positive real part is
denoted by  ${\mathcal P}$. The following   lemma for functions with
positive real part will be useful in the sequel.

\begin{lemma}\cite[Theorem 3, p.\ 80]{Good}\label{lemma1.2}
The  coefficient $c_n$ of a function $p \in {\mathcal P}$  satisfy
the sharp inequality
\begin{equation*}
|c_n|\leq 2 \quad\quad  (n \geq 1).
\end{equation*}
\end{lemma}

\section{Coefficient estimates}

In this section, certain subclasses of the class $\Sigma_{\mathcal
{B}}$ of meromorphic bi-univalent functions are introduced and
estimates on the coefficient $b_0$ and $b_1$ for functions in these
subclasses are obtained.

\begin{definition}\label{def1.1}
A function $g$ given by series expansion \eqref{eq1.2} is a
meromorphic starlike bi-univalent functions of order $\alpha$, $ 0
\leq \alpha<1 $, if
\begin{equation*}
\operatorname{Re}  \left(\frac{zg'(z)}{g(z)}\right)> \alpha \quad (z
\in \Delta),
\end{equation*}
and
\begin{equation*}
\quad\operatorname{Re}   \left(\frac{wh'(w)}{h(w)}\right)> \alpha  \
\quad (z\in\Delta),
\end{equation*}
where the function $h$ is the inverse of $g$ given by \eqref{eq1.3}.
The class of all meromorphic starlike bi-univalent functions of
order $\alpha$ is denote by $\Sigma_{\mathcal {B}}^* (\alpha)$.
\end{definition}

\begin{theorem}\label{th1.1}
If the function $g$ given by {\rm \eqref{eq1.2}} is a meromorphic
starlike bi-univalent function of order $\alpha$, $0 \leq \alpha<1$,
then the coefficients $b_0$ and $b_1$ satisfy the inequalities
\begin{equation*}
|b_0| \leq ~ 2 (1- \alpha), \quad  {\rm and} \quad  |b_1| \leq ~ (1-
\alpha) \sqrt {4 \alpha ^2 - 8 \alpha +5}.
\end{equation*}
\end{theorem}

\begin{proof} Let $g$ be the meromorphic starlike bi-univalent
function of order $\alpha$ given by \eqref{eq1.2}. Then
a calculation using Equation \eqref{eq1.2}  shows that
\begin{equation}\label{eq2.10}
\frac{zg'(z)}{g(z)}= 1- \frac{b_0}{z}+ \frac{b_0^2- 2b_1}{z^2}-
\frac{b_0^3- 3b_1b_0+ 3b_2}{z^3}+\cdots \quad (z\in\Delta).
\end{equation}
Since  $h=g^{-1}$ is the inverse of $g$ whose series expansion is
given by \eqref{eq1.3}, a computation shows that
\begin{align*}
w=g(h(w))&= (b_0+B_0)+ w + \frac{b_1+B_1}{w}+
\frac{B_2-b_1B_0+b_2}{w^2} \\
&\quad {}+ \frac{B_3 - b_1B_1 + b_1B_0 ^2 - 2b_2B_0 +
b_3}{w^3}+\cdots .
\end{align*}
Comparing the initial coefficients, the following relations are
obtained:
\begin{equation}\label{eq2.1}
b_0+B_0= 0,
\end{equation}
\begin{equation}\label{eq2.2}
b_1+B_1= 0,
\end{equation}
\begin{equation}\label{eq2.3}
B_2-b_1B_0+b_2= 0,
\end{equation}
and
\begin{equation}\label{eq2.4}
B_3 - b_1B_1 + b_1B_0 ^2 - 2b_2B_0 + b_3=0.
\end{equation}
Equations \eqref{eq2.1}--\eqref{eq2.4}  yield
\begin{equation}\label{eq2.5}
B_0 = -b_0,
\end{equation}
\begin{equation}\label{eq2.6}
B_1= -b_1,
\end{equation}
\begin{equation}\label{eq2.7}
B_2= -b_2 - b_0 b_1,\\
\end{equation}
and
\begin{equation}\label{eq2.8}
B_3=-(b_3+   2b_0b_2+b_0^2 b_1+b_1 ^2 ).
\end{equation}
Use of Equations \eqref{eq2.5}--\eqref{eq2.8} shows that  the series
expansion for the function $g^{-1}$ given by \eqref{eq1.3} becomes
\begin{equation}\label{eq2.9}
h(w)= g^{-1}(w)= w - b_0 - \frac {b_1}{w} - \frac {b_2 + b_0
b_1}{w^2} - \frac {b_3+   2b_0b_2+b_0^2 b_1+b_1 ^2}{w^3} +\cdots .
\end{equation}
A calculation using Equation \eqref{eq2.9} shows that
\begin{equation}\label{eq2.11}
\frac{wh'(w)}{h(w)}= 1+ \frac{b_0}{w}+ \frac{b_0^2 + 2b_1}{w^2}+
\frac{b_0^3 +6 b_1b_0 + 3b_2}{w^3}+\cdots \quad (z\in\Delta).
\end{equation}
Since $g$ is a bi-univalent meromorphic function of order $\alpha$,
there exist two functions $p, q$ with positive real part in $\Delta$
of the forms
\begin{equation}\label{eq2.12}
p(z)= 1+ \frac{c_1}{z} + \frac{c_2}{z^2}+ \frac{c_3}{z^3} + \cdots
\quad   (z\in\Delta)
\end{equation}
and
\begin{equation}\label{eq2.13}
~\quad q(w)= 1+ \frac{d_1}{w}+ \frac{d_2}{w^2}+
\frac{d_3}{w^3}+\cdots \quad   (z\in\Delta).
\end{equation}
such that
\begin{equation}\label{eq2.14}
\frac{zg'(z)}{g(z)}=  \alpha + (1- \alpha) p(z),
\end{equation}
and
\begin{equation}\label{eq2.15}
\frac{wh'(w)}{h(w)}=  \alpha + (1- \alpha) q(w).
\end{equation}
Use of  \eqref{eq2.12} in \eqref{eq2.14} shows that
\begin{equation}\label{eq2.16}
\frac{zg'(z)}{g(z)}= 1+ \frac{(1- \alpha) c_1}{z} + \frac{(1-
\alpha) c_2}{z^2}+ \frac{(1- \alpha) c_3}{z^3} + \cdots .
\end{equation}
In view of the Equations \eqref{eq2.10} and \eqref{eq2.16}, it is
easy to see that
\begin{equation}\label{eq2.17}
(1- \alpha) c_1 = - b_0
\end{equation}
and
\begin{equation}\label{eq2.18}
(1- \alpha) c_2 = b_0 ^2 - 2 b_1.
\end{equation}
Similarly, use of \eqref{eq2.11}, \eqref{eq2.13} in \eqref{eq2.15}
immediately yields
\begin{equation}\label{eq2.19}
(1- \alpha) d_1 = b_0
\end{equation}
and
\begin{equation}\label{eq2.20}
(1- \alpha) d_2 = b_0 ^2 + 2 b_1.
\end{equation}
Equations \eqref{eq2.17} and \eqref{eq2.19} together yields
\begin{equation*}
c_1= - d_1
\end{equation*}
and
\begin{equation}\label{eq2.21}
b_0 ^2= \frac{(1 - \alpha) ^2}{2} (c_1 ^2 + d_1 ^2).
\end{equation} Since $\operatorname{Re}p(z)>0$ in $\Delta$, the
function $p(1/z)\in\mathcal{P}$ and hence the coefficients $c_n$ and
similarly the coefficients $d_n$ of the function  $q$ satisfy the
inequality in Lemma \ref{lemma1.2} and this immediately yields  the
following estimate:
\begin{equation*}
|b_0 ^2| = \frac{(1 - \alpha) ^2}{2} |c_1 ^2 + d_1 ^2|\leq 4 (1 -
\alpha) ^2.
\end{equation*}
This readily yields the following estimate for $b_0$:
\begin{equation*}
|b_0 | \leq 2 (1 - \alpha).
\end{equation*}
The estimate $|b_0 | \leq 2 (1 - \alpha)$ also follows directly from
\eqref{eq2.17}. Using Equations \eqref{eq2.18} and \eqref{eq2.20}
yields
\begin{equation*}
b_0 ^4 - 4 b_1 ^2 = (1- \alpha)^2 c_2   d_2,
\end{equation*}
or
\begin{equation*}
4 b_1 ^2 = - (1- \alpha)^2 c_2   d_2 + b_0 ^4.
\end{equation*}
By Lemma \ref{lemma1.2}, the estimates $|c_2|=|d_2|\leq 2$ holds.
This estimate together with the estimate of $b_0$ imply that
\begin{equation*}
4 |b_1 ^2| \leq 4(1- \alpha)^2  +16(1- \alpha)^4.
\end{equation*}
Therefore
\begin{equation*}
|b_1| \leq (1- \alpha)\sqrt{ (4 \alpha ^2 - 8 \alpha +5)}.\qedhere
\end{equation*}
\end{proof}

\begin{definition}\label{def1.2}
The function $g$ given by \eqref{eq1.2} is said to belong to class
$\widetilde{\Sigma}^*_{\mathcal{B}}(\alpha)$ of bi-univalent
strongly starlike meromorphic functions of order $\alpha$, $0
<\alpha \leq 1$, if
\begin{equation*}
\left|\arg \left(\frac{zg'(z)}{g(z)}\right)\right| < \frac{\alpha
\pi}{2}\quad    (z \in \Delta ),\quad
\end{equation*}
and
\begin{equation*}
\quad\left|\arg \left(\frac{wh'(w)}{h(w)}\right)\right|<\frac{\alpha
\pi}{2} \quad   (z\in\Delta).
\end{equation*}
where the function $h$ is the inverse of the function $g$ given by
\eqref{eq1.3}.
\end{definition}

The class considered in Definition \ref{def1.1} is related to
starlikeness of order $\alpha$ and the  second subclass in
Definition \ref{def1.2} is associated with strongly starlikeness of
order $\alpha$. It should be noted that meromorphic starlike
bi-univalent  functions of order $0$ is essentially the same as
meromorphic  strongly starlike bi-univalent functions of order $1$:
$ \Sigma_{\mathcal {B}}^* (0) \equiv
\widetilde{\Sigma}^*_{\mathcal{B}}(1)$.  In view of this connection,
it should be noticed that the class
$\widetilde{\Sigma}^*_{\mathcal{B}}(\alpha)$ provides a
generalization of the class of meromorphic  starlike bi-univalent
functions in a different direction; the class $ \Sigma_{\mathcal
{B}}^* (\alpha)$ is associated with right half-planes while the
class $\widetilde{\Sigma}^*_{\mathcal{B}}(\alpha)$ associated with
sectors. It is pertinent to see that the estimates of $b_0$ and
$b_1$ in Theorem~\ref{th1.1} when $\alpha=0$ is the same as the
corresponding estimates in Theorem~\ref{theorem2.4} when $\alpha=1$.

\begin{theorem}\label{theorem2.4}
If the function $g$ given by {\rm \eqref{eq1.2}} is in the class
$\widetilde{\Sigma}^*_{\mathcal{B}}(\alpha)$, $0<\alpha \leq 1$,
then the coefficients $b_0$ and $b_1$ satisfy the inequalities
\begin{equation*}
|b_0| \leq ~ 2 \alpha, \quad  {\rm and} \quad  |b_1| \leq ~ \sqrt
{5}~\alpha^2 .
\end{equation*}
\end{theorem}

\begin{proof}
Consider the function $g \in
\widetilde{\Sigma}^*_{\mathcal{B}}(\alpha)$. Then, by definition of
the class $\widetilde{\Sigma}^*_{\mathcal{B}}(\alpha)$,
\begin{equation}\label{eq2.22}
\frac{zg'(z)}{g(z)}=  \left(p(z) \right)^{\alpha} \quad{\rm and}
\quad\frac{wh'(w)}{h(w)}= \left(q(w)\right)^{\alpha},
\end{equation}
where $p$ and $q$ are functions with positive real part in $\Delta$
and the series expansion of $p$ and $q$ are respectively given by
\begin{equation*}
p(z)= 1+ \frac{c_1}{z}+ \frac{c_2}{z^2}+ \frac{c_3}{z^3}+\cdots
\quad   \quad  (z\in\Delta),
\end{equation*}
and
\begin{equation*}
\quad q(w)= 1+ \frac{d_1}{w}+ \frac{d_2}{w^2}+
\frac{d_3}{w^3}+\cdots \quad  \quad (z\in\Delta).
\end{equation*}
A computation yields
\begin{equation}\label{eq2.23}
\left( p(z) \right)^{\alpha}= 1+ \frac{\alpha c_1}{z} +
\frac{\frac{1}{2}\alpha (\alpha-1) c_1^2 + \alpha c_2}{z^2}+
\frac{\frac{1}{6}\alpha (\alpha -1)(\alpha -2) c_1  ^3 +
\alpha(\alpha -1) c_1 c_2 + \alpha c_3}{z^3} +\cdots
\end{equation}
and, by definition of $g$,
\begin{equation*}
\frac{zg'(z)}{g(z)}= 1- \frac{b_0}{z}+ \frac{b_0^2- 2b_1}{z^2}-
\frac{b_0^3- 3b_1b_0+ 3b_2}{z^3}+\cdots \quad(z\in\Delta).
\end{equation*}
This equation with Equation \eqref{eq2.23} and first equation in
\eqref{eq2.22} yield
\begin{equation}\label{eq2.24}
\alpha c_1= -b_0,
\end{equation}
and
\begin{equation}\label{eq2.25}
\frac{1}{2}\alpha (\alpha-1) c_1^2 + \alpha c_2 = b_0^2 - 2b_1.
\end{equation}
Similarly
\begin{equation}\label{eq2.26}
\left(q(w) \right)^{\alpha}= 1+ \frac{\alpha d_1}{w} +
\frac{\frac{1}{2}\alpha (\alpha -1) d_1^2 + \alpha d_2}{w^2}+
\frac{\frac{1}{6}\alpha (\alpha -1)(\alpha -2) d_1  ^3 +
\alpha(\alpha -1) d_1 d_2 + \alpha d_3}{w^3} + \cdots
\end{equation}
and
\begin{equation*}
\frac{wh'(w)}{h(w)}= 1+ \frac{b_0}{w}+ \frac{b_0^2 + 2b_1}{w^2}+
\frac{b_0^3 +6 b_1b_0 + 3b_2}{w^3}+\cdots \quad  (z\in\Delta).
\end{equation*}
The last equation and Equation \eqref{eq2.26} together with the
second equation in \eqref{eq2.22} implies
\begin{equation}\label{eq2.27}
 \alpha d_1= b_0,
\end{equation}
and
\begin{equation}\label{eq2.28}
\frac{1}{2}\alpha (\alpha -1) d_1^2 + \alpha d_2 = b_0^2 + 2b_1.
\end{equation}
Using the Equations \eqref{eq2.24} and \eqref{eq2.27}, one gets
\begin{equation*}
c_1= - d_1
\end{equation*}
and
\begin{equation*}
2b_0 ^2= \alpha ^2 (c_1 ^2 + d_1 ^2)
\end{equation*}
which implies
\begin{equation}\label{eq2.29}
b_0 ^2= \frac{\alpha ^2}{2} (c_1 ^2 + d_1 ^2).
\end{equation}
By Lemma \ref{lemma1.2},  $|c_1|\leq 2$ and $|d_1|\leq 2$ and using
them in \eqref{eq2.29}, it follows that
\begin{equation*}
|b_0 ^2|= \frac{\alpha ^2}{2}~ |c_1 ^2 + d_1 ^2|~ \leq ~\frac{\alpha
^2}{2} (|c_1 ^2| + |d_1 ^2|) \leq  4 \alpha ^2.
\end{equation*}
Hence
\begin{equation*}
|b_0| \leq ~ 2 \alpha .
\end{equation*}
Equations \eqref{eq2.25} and \eqref{eq2.27} together yield
\begin{equation}\label{eq2.30}
2b_0^4 + 8 b_1^2= \frac{1}{4}\alpha^2 (\alpha-1)^2 (c_1^4 + d_1 ^4)
+ \alpha^2 (c_2^2 + d_2 ^2)+ \alpha^2 (\alpha-1)(c_1^2 c_2 + d_1 ^2
d_2).
\end{equation}
In view of \eqref{eq2.29}, the previous equation becomes
\begin{equation*}
b_1^2= \frac{\alpha^2 (\alpha -1)^2}{32} (c_1^4 + d_1 ^4) +
\frac{\alpha^2}{8} (c_2^2 + d_2 ^2)+ \frac{\alpha^2
\alpha-1)}{8}(c_1^2 c_2 + d_1 ^2 d_2)- \frac {\alpha ^4}{16} (c_1 ^4
+ d_1 ^4)- \frac {\alpha ^4}{8}c_1^2 d_1 ^2,
\end{equation*}
Lemma \ref{lemma1.2} again gives the estimates $|c_i|=|d_i|\leq 2$
for $i=1,2$, and using these in the  above equation immediately
yields
\begin{align*}
|b_1^2| &\leq  \alpha^2 (\alpha -1)^2 + \alpha^2+ 2 \alpha^2
(\alpha-1)+ 2 \alpha ^4 + 2 \alpha ^4 = 5 \alpha ^4.
\end{align*}
This shows that
\begin{equation*}
|b_1| \leq ~ \sqrt {5}~\alpha^2 .\qedhere
\end{equation*}
\end{proof}

\section{Meromorphic  Bazilevi\v{c} bi-univalent functions }

This section is related to a general class called the class of
 meromorphic Bazilevi\v{c} bi-univalent functions. Let $p \in \mathcal{P}$, $h \in
\mathcal{S^*}$, $\alpha$ any real number and $\beta >0$,
Bazilevi\v{c} \cite{Bazil} introduced a subclass of $\mathcal {A}$
consisting of the principal branch of the functions
\begin{equation*}
f(z)=\left(\frac{\beta}{1+ \alpha ^2}\int_0 ^z (p(\xi)-\alpha i)
\xi^{-\frac{\alpha \beta i }{1 + \alpha ^2}-1} h(\xi)^\frac{\beta}{1
+ \alpha ^2} d\xi \right)^\frac{1 + \alpha i}{\beta}
\end{equation*}
and he showed that each  principal branch   is univalent in
$\mathbb{D}$. In the case when $\alpha=0$, a computation shows that
\begin{equation*}
zf'(z)= f(z) ^{1- \beta} h(z) ^{\beta} p(z)
\end{equation*}
or
\begin{equation}\label{eq3.1}
\operatorname{Re}
\left(\frac{zf'(z)}{f(z)^{1-\beta}h(z)^\beta}\right) >0.
\end{equation}
Thomas \cite{Thomas} called a function satisfying the condition
\eqref{eq3.1} as a Bazilevi\v{c} function of type $\beta$.
Furthermore, if $h(z)=z$ in \eqref{eq3.1}, then the condition
\eqref{eq3.1} becomes
\begin{equation}\label{eq3.2}
\operatorname{Re} \left(\frac{zf'(z)}{f(z)^{1-\beta}z^\beta}\right)
>0.
\end{equation}
The class of all functions $f \in A$ satisfies \eqref{eq3.2} is
introduced by Singh \cite{Singh} and the class of all such functions
is denoted by $B(\beta)$. In this section, the estimates for the
initial coefficients of the meromorphic functions  analogous to the
functions belonging to the  class $B(\beta)$ are obtained.

\begin{definition}Let $\beta>0$ and $0< \alpha \leq 1$.
A meromorphic bi-univalent  function $g$ given by \eqref{eq1.2} is
said to be in the class $\Sigma^{B}_{\mathcal{B}}(\beta, \alpha)$ of
meromorphic strongly Bazilevi\v{c} bi-univalent functions  of type
$\beta$ and order $\alpha$,  if
\begin{equation*}
\left|\arg
\left(\left(\frac{z}{g(z)}\right)^{1-\beta}g'(z)\right)\right|<\frac{\alpha
\pi}{2}\quad  (z\in\Delta)
\end{equation*}
and
\begin{equation*}
\quad  \left|\arg
\left(\left(\frac{w}{h(w)}\right)^{1-\beta}h'(w)\right)\right|<\frac{\alpha
\pi}{2} \quad (z\in\Delta)
\end{equation*}
where the function $h$ is the inverse of $g$ and given by
\eqref{eq1.3}.
\end{definition}

\begin{theorem}
Let $\beta>0$ and $0<\alpha \leq 1$. If $g \in
\Sigma^{B}_{\mathcal{B}}(\beta, \alpha)$, then the coefficients
$b_0$ and $b_1$ satisfy the inequalities
\begin{equation*}
|b_0|  \leq  \frac{2 \alpha }{1 - \beta}, \quad  {\rm and} \quad
|b_1| \leq \frac{2 \alpha ^2}{(1- \beta)(2- \beta)} \sqrt {2 (1-
\beta)(2- \beta) + 1}.
\end{equation*}
\end{theorem}

\begin{proof}
Suppose $g \in \Sigma^{B}_{\mathcal{B}}(\beta, \alpha)$ has a
representation given by \eqref{eq1.2}, then a computation shows that
\begin{align*}
g'(z) & = 1- \frac{b_1}{z^2}- \frac{2b_2}{z^3}- \frac{3b_3}{z^4}+
\cdots \intertext{and} \frac{z}{g(z)}&= 1-\frac{b_0}{z}+ \frac{b_0^2
- b_1}{z^2} - \frac{b_0^3 - 2b_1b_0 + b_2}{z^3} + \cdots .
\intertext{Furthermore} \left(\frac{z}{g(z)} \right)^{1- \beta} &=
1- \frac{(1- \beta)b_0}{z}
+ \frac{(1- \beta)\left((2- \beta)b_0^2- 2b_1\right)}{2z^2}\\
&\quad{}+  \frac{(1- \beta)\left( \left( (1- \beta)^2 + 3 (1- \beta)
+2\right)b_0^3 - 6(2- \beta) b_1b_0 + 6b_2 \right)}{6z^3}+ \cdots.
\end{align*}
Further calculations show that
\begin{align}
\left(\frac{z}{g(z)} \right)^{1-\beta}  g'(z)
&= 1- \frac{(1-\beta) b_0}{z} + \frac{(2-\beta)\left((1-\beta) b_0^2 -2 b_1\right)}{2z^2} \notag\\
&\quad {}- \frac{(3-\beta) \left((1-\beta) (2-\beta)b_0^3 + 6(1-\beta)~ b_1 b_0 + 6b_2\right)}{6z^3}\nonumber\\
&\quad{}+  \frac{1- 12(1-\beta) \left( (2-\beta)b_1 b_0^2 - 2b_1^2
-2b_0b_2 \right)}{24z^4}+ \cdots .\label{eq3.3}
\end{align}
The assumption $g \in \Sigma^{B}_{\mathcal{B}}(\beta, \alpha)$ shows
that there is  a function $p$ with $\operatorname{Re}\left(
p(z)\right)>0$ such that
\begin{equation}\label{eq3.4}
\left(\frac{z}{g(z)}\right)^{1-\beta}g'(z)= \left( p(z)
\right)^{\alpha},
\end{equation}
where the function $p$ has the representation given by
\begin{equation*}
p(z)= 1+ \frac{c_1}{z}+ \frac{c_2}{z^2} + \frac{c_3}{z^3}+\cdots .
\end{equation*}
The Equations \eqref{eq3.3}, \eqref{eq3.4} and \eqref{eq2.23}
together yield the following:
\begin{equation}\label{eq3.5}
-(1-\beta) b_0= \alpha c_1
\end{equation}
and
\begin{equation}\label{eq3.6}
\frac{1}{2}  (2-\beta)\left((1-\beta) b_0^2 - 2b_1\right) =
\frac{1}{2} \alpha(\alpha -1) c_1^2 + \alpha c_2.
\end{equation}
Similarly
\begin{equation*}
h'(w)= 1+ \frac{b_1}{w^2}+ \frac{2(b_2 + b_0 b_1)}{w^3}+ \frac{3
\left(    b_0 ^2b_1 + 2b_0b_2 + b_1^2 + b_3 \right)}{w^4}+ \cdots
\end{equation*}
and
\begin{equation*}
\frac{w}{h(w)}= 1+ \frac{b_0}{w}+ \frac{b_0^2 + b_1}{w^2}+ \frac{
b_0^3 + 3b_1 b_0+ b_2}{w^3}+ \cdots .
\end{equation*}
Hence
\begin{align*}
\left(\frac{w}{h(w)}\right)^{1-\beta}&=  1 + \frac{(1-\beta)b_0}{w}
+ \frac{(1- \beta)\left((2- \beta)b_0^2+ 2b_1\right)}{2w^2}\\
&\quad +  \frac{(1- \beta)\left(   (2- \beta)(3- \beta) b_0^3 + 6
(2- \beta)b_1b_0 + 6b_2 \right)}{6z^3}+ \cdots
\end{align*}
and
\begin{align}
\left(\frac{w}{h(w)}\right)^{1-\beta} h'(w)
&= 1+ \frac{(1-\beta) b_0}{w} + \frac{(2-\beta)((1-\beta) b_0^2 + 2 b_1)}{2w^2} \\
& \quad {}+ \frac{(3-\beta) ( (1-\beta)(2-\beta)b_0^3 + 6(2-\beta)b_0 b_1 + 6 b_2)}{6w^3}\nonumber\\
&\quad {} + \frac{1- 12(1-\beta) \left( (2-\beta)b_1 b_0^2 - 2b_1^2
-2b_0b_2 \right)}{24w^4}+ \cdots .\label{eq3.7}
\end{align}
The hypothesis $g \in \Sigma^{B}_{\mathcal{B}}(\beta, \alpha)$
again implies that there exist a function $q $ with
$\operatorname{Re}\left( q(w)\right)>0$ satisfying
\begin{align}\label{eq3.8}
\left(\frac{w}{h(w)}\right)^{1-\beta}h'(w)= \left( q(w)
\right)^{\alpha},
\end{align}
where  $q$ has a series representation given by
\begin{equation*}
q(w)= 1+ \frac{d_1}{w}+ \frac{d_2}{w^2}+ \frac{d_3}{w^3}+\cdots .
\end{equation*}
Equations \eqref{eq3.7}, \eqref{eq3.8}    and \eqref{eq2.26} yield
\begin{equation}\label{eq3.9}
(1-\beta) b_0= \alpha d_1
\end{equation}
and
\begin{equation}\label{eq3.10}
\frac{1}{2} \left(2-\beta)((1-\beta) b_0^2 + 2b_1\right) =
\frac{1}{2} \alpha(\alpha -1) d_1^2 + \alpha d_2.
\end{equation}
Equations \eqref{eq3.5} and \eqref{eq3.9} shows that
\begin{equation*}
c_1 = - d_1
\end{equation*}
and
\begin{equation*}
2(1-\beta)^2 b_0^2 = \alpha ^2 (d_1 ^2 + c_1 ^2),
\end{equation*}
or
\begin{equation}\label{eq3.11}
b_0^2 = \frac{\alpha ^2}{2(1-\beta)^2}(d_1 ^2 + c_1 ^2).
\end{equation}

By Lemma \ref{lemma1.2}, $|c_1|=|d_1|\leq 2$  and use of this
ineqaulity in the Equation \eqref{eq3.11} immediately leads to the
following estimate for $b_0$:
\begin{align*}
|b_0^2| & = \frac{\alpha ^2}{2(1-\beta)^2}  |d_1 ^2 + c_1 ^2| \\
&\leq \frac{\alpha ^2}{2(1-\beta)^2}( |d_1 ^2| + | c_1 ^2|)\\&  =
\frac{ 4\alpha ^2}{(1-\beta)^2} .\end{align*} This completes the
proof of the inequality $|b_0| \leq 2 \alpha /(1 - \beta)$. Yet
another calculation using \eqref{eq3.6} and \eqref{eq3.10} shows
that
\begin{equation*}
\frac{1}{4}(2-\beta)^2\left((1-\beta)^2 b_0^4 -4 b_1^2 \right) =
\frac{\alpha ^2(\alpha -1)^2}{4}(d_1^2c_1 ^2) + \frac{\alpha ^2
(\alpha -1)}{2}(c_1^2d_2 + d_1^2 c_2)+ \alpha ^2c_2 d_2.
\end{equation*}
Use of \eqref{eq3.11} in the above equation leads to the following
expression for $b_1$:
\begin{align*}
-(2-\beta)^2 b_1^2 &=  \frac{\alpha ^2(\alpha -1)^2}{4}(d_1^2c_1 ^2)
+ \frac{\alpha ^2 (\alpha -1)}{2}(c_1^2d_2 + d_1^2 c_2)\\
&\quad{}+ \alpha ^2c_2 d_2 - \frac{4(2-\beta)^2\alpha
^4}{(1-\beta)^2}.
\end{align*}
Once again, an application of Lemma \ref{lemma1.2}  immediately
yields
\begin{align*}
|b_1^2| &\leq   \frac{4 \alpha ^2(\alpha -1)^2}{(2-\beta)^2}+
\frac{8 \alpha ^2 (\alpha -1)}{ (2-\beta)^2}
+ \frac {4 \alpha ^2}{(2-\beta)^2}  + \frac{4 \alpha ^4 }{(1-\beta)^2}\\
&= \frac{4 \alpha ^4 (2 (1- \beta)(2- \beta) +
1)}{(1-\beta)^2(2-\beta)^2}.
\end{align*}
and  therefore
\begin{equation*}
|b_1|  \leq    \frac{2 \alpha ^2}{(1- \beta)(2- \beta)} \sqrt {2 (1-
\beta)(2- \beta) + 1}.  \qedhere
\end{equation*}
\end{proof}

\begin{remark}
If $b_0=0$  for the function $g \in \Sigma$, the series expansion
\eqref{eq2.9} becomes
\begin{equation*}
g^{-1}(w)= w - \frac {b_1}{w} - \frac {b_2}{w^2} - \frac {b_1 ^2 +
b_3}{w^3} +\cdots
\end{equation*}
This series expansion was obtained by Schober \cite{Schober}.
\end{remark}
\begin{example}
The function $g(z)= z + 1/z$ is clearly a univalent meromorphic
function. A direct calculation that
\begin{equation*}
g^{-1}(w)= \frac {w+\sqrt{w^2 - 4}}{2}.
\end{equation*}
This function shows $g^{-1}$ has the series expansion given by
\begin{equation*}
g^{-1}(w)= w - \frac{1}{w} - \frac{1}{w^3} - \frac{2}{w^5} -
\frac{5}{w^7} - \frac{14}{w^9} -\cdots.
\end{equation*}
\end{example}

\begin{theorem}
If $g$ given by {\rm \eqref{eq1.2}} is in the class
$\Sigma^*_{\mathcal{B}}(\alpha)$, $0<\alpha \leq 1$, and $b_0=0$,
then
\begin{equation*}
|b_1| \leq   \alpha.
\end{equation*}
\end{theorem}

\begin{proof}
Assume that the function $g=z+ \sum _{n=1}^{\infty} b_n z^{-n} \in
\Sigma^*_{\mathcal{B}}(\alpha)$ where $0<\alpha \leq 1$. Since
$b_0=0$, $c_1 = d_1 =0$ and the result can be verified by a direct
calculation of \eqref{eq2.16}.
\end{proof}

\begin{theorem}
Let $g \in \widetilde{\Sigma}^{B}_{\mathcal{B}}(\alpha, \beta)$,
where $\alpha>0$ and $0<\beta \leq 1$. Then
\begin{equation*}
|b_1| \leq \frac{2 \beta ^2}{2- \alpha}.
\end{equation*}
\end{theorem}

\begin{proof}
Since the function $g=z+ \sum _{n=1}^{\infty} b_n z^{-n} \in
\widetilde{\Sigma}^{B}_{\mathcal{B}}(\alpha, \beta)$ where $0<\alpha
\leq 1$ and $b_0=0$, it follows that $c_1 = d_1 =0$. By replacing
these values in Equation \eqref{eq2.30} and continuing as in the
proof of Theorem \ref{theorem2.4}, the result is obtained.
\end{proof}

 \end{document}